\documentclass[a4paper, 11pt]{amsart}
\usepackage[T1]{fontenc}
\usepackage[utf8]{inputenc}
\usepackage[UKenglish]{babel}
\usepackage{amsmath}
\usepackage{amssymb}
\usepackage{amsfonts}
\usepackage{amsthm}
\usepackage{bm}
\usepackage{graphicx}
\usepackage{enumitem}
\usepackage{moreenum}
\usepackage{mathtools}
\usepackage{csquotes}
\usepackage{hyperref}

\numberwithin{equation}{section}
\newtheorem{theorem}{Theorem}[section]
\newtheorem{lemma}[theorem]{Lemma}
\newtheorem{proposition}[theorem]{Proposition}
\newtheorem{remark}[theorem]{Remark}

\newcommand{\per}{\mathrm{per}}
\newcommand{\e}{\mathrm{e}}
\newcommand{\R}{\mathbb R}
\newcommand{\N}{\mathbb N}
\newcommand{\I}{\mathcal I}
\newcommand{\F}{\mathcal F}
\newcommand{\SL}{\mathcal S}
\newcommand{\M}{\mathcal M}
\newcommand{\A}{\mathcal A}
\newcommand{\U}{\mathcal U}
\newcommand{\Cc}{\mathcal C}

\allowdisplaybreaks

\title[Axisymmetric capillary water waves connecting to static unduloids]{Axisymmetric capillary water waves with vorticity and swirl connecting to static unduloid configurations}

\author{Anna-Mariya Otsetova}
\address{Anna-Mariya Otsetova\newline Department of Mathematics and Systems Analysis, Aalto University, P.O. Box 11100, 00076 Aalto, Finland}
\email{anna-mariya.otsetova@aalto.fi}

\author{Erik Wahlén}
\address{Erik Wahlén\newline Centre for Mathematical Sciences, Lund University, P.O. Box 118, 22100 Lund, Sweden}
\email{erik.wahlen@math.lu.se}

\author{Jörg Weber}
\address{Jörg Weber\newline Faculty of Mathematics, University of Vienna, Oskar-Morgenstern-Platz 1, 1090 Vienna, Austria}
\email{joerg.weber@univie.ac.at}



\begin{document}
\begin{abstract}
We study steady axisymmetric water waves with general vorticity and swirl, subject to the influence of surface tension. Explicit solutions to such a water wave problem are static configurations where the surface is an unduloid, that is, a periodic surface of revolution with constant mean curvature. We prove that to any such configuration there connects a global continuum of non-static solutions by means of a global implicit function theorem. To prove this, the key is strict monotonicity of a certain function describing the mean curvature of an unduloid and involving complete elliptic integrals. From this point of view, this paper is an interesting interplay between water waves, geometry, and properties of elliptic integrals.
\end{abstract}
\maketitle

\section{Introduction}
We study axisymmetric water waves (waves travelling at constant speed on the surface of a fluid jet) with surface tension, modelled by assuming that the domain is bounded by a free surface on which capillary forces are acting, and that in cylindrical coordinates $(r,\vartheta,z)$ the domain and flow are independent of the azimuthal variable $\vartheta$. In the irrotational and swirl-free setting, such waves were studied numerically by Vanden-Broeck et al.~\cite{VandenBroeckMilohSpivack98} and Osborne and Forbes \cite{OsborneForbes01}, who found similarities to two-dimensional capillary waves, including overhanging profiles and limiting configurations with trapped bubbles at their troughs. The main motivation of the present paper is to make rigorous the numerical observation in \cite{VandenBroeckMilohSpivack98} that as another possible limiting behaviour static configurations are approached, that is, configurations with no flow at all; see \cite[Figure 6]{VandenBroeckMilohSpivack98}. While the classic way, in order to compute solutions numerically or construct them rigorously, is to bifurcate from laminar flows and then to continue along the bifurcation branch up to a limiting configuration, our approach goes in the opposite way: We start from a static, non-flat configuration and then construct rigorously a global continuum of solutions connecting to this static configuration. We point out that, for reasons we explain later, our construction is by means of the (global) implicit function theorem and \textit{not} by means of bifurcation methods.

It is well-known \cite{ErhardtWahlenWeber22,Saffman92} that the capillary axisymmetric water wave problem can be conveniently written in terms of the Stokes stream function $\Psi$, which gives rise to the velocity $\bm{u}$ through
\[\bm{u}=\frac{F(\Psi)}{r}\bm{e}_\vartheta-\nabla\times(\Psi\bm{e}_\vartheta/r)=\frac{\Psi_z}{r}\bm{e}_r+\frac{F(\Psi)}{r}\bm{e}_\vartheta-\frac{\Psi_r}{r}\bm{e}_z.\]
Here, the swirl $F$ is an arbitrary function of $\Psi$ and $(\bm{e}_r,\bm{e}_\vartheta,\bm{e}_z)$ is the usual basis with respect to cylindrical coordinates. With this at hand, we look for periodic solutions to the problem
\begin{subequations}\label{eq:OriginalEquation_Psi}
	\begin{align}
		\Delta^*\Psi&=-r^2\gamma(\Psi)-F(\Psi)F'(\Psi) &&\text{in }\Omega,\label{eq:OriginalEquation_Psi_PDE}\\
		\frac{\Psi_r^2+\Psi_z^2+F(\Psi)^2}{2r^2}-\sigma\kappa&=Q && \text{on }\partial\Omega_\mathcal{S},\\
		\Psi&=m &&\text{on }\partial\Omega_\mathcal{S},\\
		\Psi&=0 &&\text{on }\partial\Omega_\mathcal{C}.\label{eq:Psi_r=0}
	\end{align}
\end{subequations}
Here, $Q$ and $m$ are constants, $\sigma>0$ is the constant coefficient of surface tension, $\gamma$ is yet another arbitrary function of $\Psi$, and
\begin{align}\label{eq:Grad-Shafranov_op}
	\Delta^*\coloneqq\partial_r^2-\frac1r\partial_r+\partial_z^2
\end{align}
is the Grad--Shafranov operator. Moreover, $\Omega=\{(r,z)\in\R^2:0<r<\eta(z)\}$ is one half of one cross-section of the fluid domain, where to find $\eta$ is part of the problem. We denote its boundaries by $\partial\Omega_\mathcal{S}=\{(r,z)\in\R^2:r=\eta(z)\}$ (the free surface) and $\partial\Omega_\mathcal{C}=\{(r,z)\in\R^2:r=0\}$ (the centre line). Although the latter could be considered as part of the domain, it is sometimes convenient to consider it as a boundary due to the appearance of inverse powers of $r$ in the equations. Finally, $\kappa$ is the mean curvature of the surface of the fluid domain in $\R^3$, which results from rotating $\Omega$ around the cylinder axis $r=0$. The PDE \eqref{eq:OriginalEquation_Psi_PDE} appears in many applications in physics and is therefore known under various names: \textit{Hicks equation}, \textit{Bragg--Hawthorne equation}, \textit{Squire--Long equation}, or \textit{Grad--Shafranov equation}. While $F(\Psi)$ is the swirl, the vorticity vector $\bm{\omega}$ is in terms of $\Psi$, $\gamma$, and $F$ given by
\[\bm{\omega}=-\frac{1}{r}F'(\Psi)\Psi_z\bm{e}_r-\left(r\gamma(\Psi)+\frac{(FF')(\Psi)}{r}\right)\bm{e}_\vartheta+\frac{1}{r} F'(\Psi)\Psi_r\bm{e}_z.\]
For a much more thorough introduction to this problem and these equations we refer to \cite{ErhardtWahlenWeber22}.

The only regularity assumption that we shall impose on $\gamma$ and $F$ throughout is
\begin{align}\label{ass:regularity}
	\gamma\in C_{\text{loc}}^{1,1}(\R),\quad F\in C_{\text{loc}}^{2,1}(\R),\quad\|\gamma'\|_\infty<\infty,\quad\|(FF')'\|_\infty<\infty,
\end{align}
where here and in the following $C^{k,\alpha}$ ($0<\alpha\le1$, $k\in\N_0$) denotes Hölder spaces (Lipschitz if $\alpha=1$) as usual. (Compared to \cite{ErhardtWahlenWeber22} we need one derivative of $\gamma$ and $F$ less. This is because we shall apply the implicit function theorem in this paper, requiring only $C^1$-regularity of the nonlinear operator, while \cite{ErhardtWahlenWeber22} uses the Crandall--Rabinowitz theorem, requiring $C^2$-regularity of the nonlinear operator.)

Let us now present the plan of the paper. In Section \ref{sec:Unduloids} we investigate more closely static configurations, which exist provided
\begin{align}\label{ass:gamma(0),F(0)}
	\gamma(0)=F(0)=0
\end{align}
and the surface is a so-called \textit{unduloid}. These surfaces are periodic surfaces of revolution with constant mean curvature and are classical in geometry. We present some well-known facts about unduloids and then introduce a new set of two parameters, which classifies all unduloids and such that one parameter corresponds to its \enquote{size} and the other to its \enquote{shape}, while the corresponding period, importantly, only depends on the \enquote{size} parameter. The key property, which is crucial for the main theorem and the cornerstone of this paper, is \eqref{eq:kappa_monotone}. It says that the (constant) mean curvature of an unduloid, given in terms of complete elliptic integrals, is strictly monotone in the \enquote{shape} parameter when the \enquote{size} parameter is kept fixed. Then, in Section \ref{sec:Reformulation}, we present a reformulation of the equation, including the usual procedure of flattening the domain, in order to put the subsequent analysis on firm ground. Also, we reformulate the equations in the form \enquote{identity plus compact}, desiring the application of a \textit{global} implicit function theorem later on. Most of Section \ref{sec:Reformulation} is very similar to \cite{ErhardtWahlenWeber22}, so we shall only present the main ingredients, while referring the reader to \cite{ErhardtWahlenWeber22} for more details. Finally, in Section \ref{sec:MainResult} we state and prove the main result. For this result, we shall, except for the regularity assumption \eqref{ass:regularity} and the condition \eqref{ass:gamma(0),F(0)}, assume that the PDE operator corresponding to \eqref{eq:OriginalEquation_Psi_PDE} linearised at a given static configuration has trivial kernel. This additional assumption is certainly not true for all choices of $\gamma$, $F$, and the static configuration, but easily seen to be true for example in case $\gamma'(0)\le0$, $F'(0)=0$, and therefore, in particular, in the irrotational, swirl-free setting; see Proposition \ref{prop:cond_zero_eigenvalue}. Thus, our main result holds true in the case numerically studied in \cite{VandenBroeckMilohSpivack98}, and the present paper puts their observation on firm rigorous ground. Loosely speaking -- for a more precise statement we refer to Theorem \ref{thm:MainResult} -- our main result says that, under the above-mentioned conditions, the implicit function theorem can be applied at a given static solution, yielding a local curve of non-static solutions to \eqref{eq:OriginalEquation_Psi}. In fact, this curve can be extended to a global continuum of solutions where the continuum either
\begin{itemize}
	\item is unbounded, or
	\item loops back to the initial static configuration, or
	\item approaches a degenerate configuration, where the surface intersects the cylinder axis.
\end{itemize}
As customary for capillary problems, we cannot, unfortunately, eliminate the latter two alternatives. In the proof of the main theorem everything is reduced to studying the kernel of a certain linear second order ODE operator. A basis of this kernel can be constructed explicitly such that the first basis element is periodic, but odd, while the second is even, but not periodic. Consequently, the operator has trivial kernel viewed as acting on even and periodic functions, which is why we are then within the scope of the implicit function theorem. Notice, importantly, that in order to construct this second basis element the key property \eqref{eq:kappa_monotone} is crucially made use of -- without \eqref{eq:kappa_monotone} the proof of our result would simply not work!

This paper has its origins in a master’s thesis \cite{Otsetova23} written by the first author under the supervision of the other two authors. Note, however, that \cite{Otsetova23} only contains a local continuation result and is restricted to irrotational flow.

\section{Unduloids}\label{sec:Unduloids}
An important observation is that \eqref{eq:OriginalEquation_Psi} gives rise to \textit{static} solutions in the sense that $\Psi\equiv0$ solves \eqref{eq:OriginalEquation_Psi} (with $m=0$ and $Q=-\sigma\kappa$), provided \eqref{ass:gamma(0),F(0)} and
\[\kappa=\text{const.}\]
While \eqref{ass:gamma(0),F(0)} is just a constraint on the functions $\gamma$ and $F$ that we shall impose throughout, the characterisation of surfaces of revolution with constant mean curvature is classical in geometry and goes back to Delauney \cite{Delaunay}. Of special interest to us are the so-called \textit{unduloids}, which are periodic surfaces of revolution. A corresponding graph (that is rotated around the cylinder axis) is obtained by tracing out one focus of an ellipse as it rolls along the axis \cite{Eells}. A more modern discussion of unduloids can be found for example in \cite{HadzhilazovaMladenovIvailoOprea07}. Depending on two parameters $a$ and $c$ with $c>a>0$, any unduloid (up to translations in $z$) can be constructed by rotating the parametric curve $\R\ni t\mapsto(r,z)(t)$,
\begin{subequations}\label{eq:unduloid_parametric}
\begin{align}
	r(t)&=r_{a,c}(t)=\sqrt{p\sin(\mu t)+q},\\
	z(t)&=z_{a,c}(t)=aG\left(\frac{\mu t}{2}-\frac{\pi}{4},k\right)+cE\left(\frac{\mu t}{2}-\frac{\pi}{4},k\right),
\end{align}
\end{subequations}
around the $z$-axis. Here,
\begin{subequations}\label{eq:mu,p,q,k}
\begin{gather}
	\mu\coloneqq\frac{2}{a+c},\quad p\coloneqq\frac{c^2-a^2}{2},\quad q\coloneqq\frac{c^2+a^2}{2},\\
	k\coloneqq\frac{\sqrt{c^2-a^2}}{c},\label{eq:k}
\end{gather}
\end{subequations}
and $G$ (usually called $F$, but the letter $F$ is already used for the swirl in this paper) and $E$ are incomplete elliptic integrals of the first and second kind, respectively,
\[G(\phi,k)=\int_0^\phi\frac{du}{\sqrt{1-k^2\sin^2u}},\quad E(\phi,k)=\int_0^\phi\sqrt{1-k^2\sin^2u}\,du,\quad\phi\in\R,\; 0<k<1.\]
The period $P_{a,c}$ of the obtained unduloid (in the sense of periodicity along the $z$-axis) is 
\begin{align}\label{eq:period_ac}
	P_{a,c}=2cE(k)+2aK(k),
\end{align}
where $K(k)$ and $E(k)$ are complete elliptic integrals of the first and second kind, respectively,
\[K(k)=G(\tfrac\pi2,k)=\int_0^{\frac\pi2}\frac{du}{\sqrt{1-k^2\sin^2u}},\quad E(k)=E(\tfrac\pi2,k)=\int_0^{\frac\pi2}\sqrt{1-k^2\sin^2u}\,du.\]
Moreover, the (constant) mean curvature of the unduloid is given by 
\begin{align}\label{eq:meancurvature_ac}
	\kappa=\kappa_{a,c}=-\frac{2}{a+c}.
\end{align}
Let us remark that our convention is that a flat cylinder with radius $R>0$ has mean curvature $-1/R$, which explains the factor $-2$ in \eqref{eq:meancurvature_ac}. While this convention is not the one usually used in geometry (the factor $-2$ being absent), it is consistent with the literature on water waves, where the same convention is also used in \cite{ErhardtWahlenWeber22,VandenBroeckMilohSpivack98}; see also \eqref{eq:kappa} below.

For our purposes, it will be convenient to write the curve describing the unduloid not in the parametric form \eqref{eq:unduloid_parametric}, but in a graph form $r=r(z)$. Indeed, this can be done since both $G(\cdot,k)$ and $E(\cdot,k)$ are clearly strictly monotone. Let us call the corresponding description of that curve $r=\rho_{a,c}(z)$ for some function $\rho_{a,c}$ which is $P_{a,c}$-periodic in $z$. Clearly, $\rho_{a,c}$ is even in $z$ since, in \eqref{eq:unduloid_parametric}, $z=0$ corresponds to $t=\frac{\pi}{2\mu}$, and, with respect to $t=\frac{\pi}{2\mu}$, $r(t)$ is even while $z(t)$ is odd.

Having in mind that ultimately we want to work with functions with a \textit{fixed} period, but $P_{a,c}$ depends on both parameters $a$ and $c$, we strive to use some other parameters $b$ and $k$. These should be related to $a$ and $c$ such that $b$ determines the \enquote{size} and $k$ the \enquote{shape} of the unduloid, while variations in only $k$ leave the corresponding period $P=P(b)$ unchanged (the motivation for this will become more apparent in the proof of Lemma \ref{lma:T_injective}). Looking at how the formula \eqref{eq:unduloid_parametric} depends on $a$ and $c$, it comes at no surprise that for the \enquote{shape} parameter exactly the $k$ already introduced above can be used, since $k=\sqrt{1-(a/c)^2}\in(0,1)$ is in one-to-one correspondence with $a/c\in(0,1)$, which clearly determines the \enquote{shape} of the curve given in \eqref{eq:unduloid_parametric}. Now observe that \eqref{eq:period_ac} can also be written as
\begin{align}\label{eq:period_ac2}
	P_{a,c}=2c\left(E(k)+\sqrt{1-k^2}K(k)\right).
\end{align}
Thus, the natural definition for the \enquote{size} parameter $b$, in order for the period $P$ to be only dependent on $b$, is
\begin{align}\label{eq:b}
	b\coloneqq\frac{c}{g(k)}
\end{align}
with
\[g(k)\coloneqq\frac{\pi}{E(k)+\sqrt{1-k^2}K(k)},\]
normalised such that $b=1$ will correspond to period $2\pi$. It is now easy to check that \eqref{eq:k} and \eqref{eq:b} define a bijection between the parameters $(a,c)$ (satisfying $c>a>0$) and $(b,k)$ (satisfying $b>0$, $k\in(0,1)$), with inverse relations
\begin{align}\label{eq:inverse_ca_bk}
	c=bg(k),\quad a=c\sqrt{1-k^2}=bg(k)\sqrt{1-k^2}.
\end{align}
We can also rewrite the surface profile in terms of $(b,k)$. A closer look at \eqref{eq:unduloid_parametric} and \eqref{eq:mu,p,q,k} reveals that in fact
\[\rho_{a,c}=c\rho_{a/c,1}(\cdot/c).\]
With this and \eqref{eq:inverse_ca_bk} in mind, we see that $\rho_{a,c}=\eta^{b,k}$, where
\begin{align}\label{eq:eta_bk}
	\eta^{b,k}\coloneqq bg(k)\rho_{\sqrt{1-k^2},1}\left(\frac{\cdot}{bg(k)}\right).
\end{align}
Also, from combining \eqref{eq:eta_bk} and \eqref{eq:period_ac2} we see that $\eta^{b,k}$ is periodic with period $2\pi b$ only depending on $b$ as desired. Notice also that all $\eta^{b,k}$ are even functions of $z$.

In Figure \ref{fig:unduloid} the unduloid that is determined by a certain such profile $\eta^{b,k}$ is plotted. Let us remark here that, while we are only interested in $k\in(0,1)$ throughout, one can also say what happens in the limit $k\to0$ or $k\to1$. Indeed, $k=0$ corresponds to a flat cylinder with radius $b$, while $k=1$ corresponds to periodically repeated spheres with radius $\pi b$. Why do we not include these limiting configurations in our analysis? On the one hand, we comment on the qualitatively different case $k=0$ in Remark \ref{rem:k0=0}. On the other hand, periodically repeated spheres are already singular configurations: At the point where they touch each other they intersect the cylinder axis, and geometric quantities like mean curvature are not well-defined there.

\begin{figure}
	\includegraphics{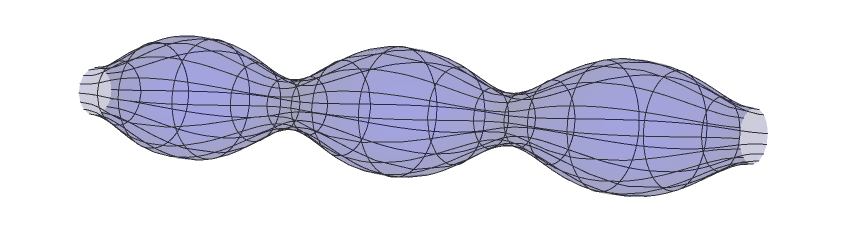}
	\caption{The unduloid corresponding to $\eta^{b,k}$ for parameter value $k=\sqrt{0.19}$. Notice that, since no scales are specified in the plot, the parameter $b$ is also left unspecified as it does not affect the shape of the unduloid.}
	\label{fig:unduloid}	
\end{figure}

Now importantly, by \eqref{eq:meancurvature_ac}, the unduloid corresponding to $\eta^{b,k}$ has (constant) mean curvature
\[\kappa^{b,k}=-\frac{2}{bg(k)(1+\sqrt{1-k^2})}=-2\cdot\frac{E(k)+\sqrt{1-k^2}K(k)}{\pi b(1+\sqrt{1-k^2})}.\]
Let us state already here what is in fact the crucial observation in this paper and what causes a certain linearised operator later to be invertible.
\begin{lemma}\label{lma:kappa_monotone}
	For $b>0$ and $k\in(0,1)$, we have
	\begin{align}\label{eq:kappa_monotone}
		\quad\partial_k\kappa^{b,k}\ne0.
	\end{align}
\end{lemma}
\begin{proof}
	From the well-known formulas \cite{Lawden89}
	\begin{gather*}
		K'(k)=\frac{E(k)}{k(1-k^2)}-\frac{K(k)}{k},\quad E'(k)=\frac{E(k)-K(k)}{k},\\
		K(k)=\frac{\pi}{2}\sum_{n=0}^{\infty}\left(\frac{(2n)!}{2^{2n}(n!)^2}\right)^2k^{2n},\quad E(k)=\frac{\pi}{2}\sum_{n=0}^{\infty}\left(\frac{(2n)!}{2^{2n}(n!)^2}\right)^2\frac{k^{2n}}{1-2n}
	\end{gather*}
	we infer by direct computation that
	\[\partial_k\kappa^{b,k}=-2\cdot\frac{2(E(k)-K(k))+k^2K(k)}{\pi bk\sqrt{1-k^2}(1+\sqrt{1-k^2})},\]
	where
	\[2(E(k)-K(k))+k^2K(k)=-2\pi\sum_{n=2}^{\infty}\left(\frac{(2n!)}{2^{2n}(n!)^2}\right)^2\frac{n(n-1)}{(2n-1)^2}k^{2n}<0.\]
\end{proof}

\section{Reformulation of the equations}\label{sec:Reformulation}
Let us now state a reformulation of the original equations \eqref{eq:OriginalEquation_Psi} to which we will later apply a global implicit function theorem. For the most part, this reformulation was presented and explained in detail in \cite{ErhardtWahlenWeber22}, so we only provide a summary here, but explain some very minor deviations from \cite{ErhardtWahlenWeber22}. These are due to the fact that in \cite{ErhardtWahlenWeber22} the mean $\langle\eta\rangle$ of the surface profile $\eta$ is fixed while the Bernoulli constant $Q$ varies, whereas in this paper we fix $Q$ and allow for varying $\langle\eta\rangle$.
\subsection{Ni's trick}
In order to get around the apparent coordinate singularity $1/r$ appearing in \eqref{eq:OriginalEquation_Psi_PDE} through the operator \eqref{eq:Grad-Shafranov_op}, a common trick going back to Ni \cite{Ni} is to consider $\psi$ instead of $\Psi$, related to each other via
\[\Psi=r^2\psi,\]
which on its own makes \eqref{eq:Psi_r=0} redundant. The advantage of $\psi$ is that it solves
\begin{subequations}
	\begin{align}
		\psi_{rr}+\frac3r\psi_r+\psi_{zz}&=-\gamma(r^2\psi)-\frac{1}{r^2}(FF')(r^2\psi)&&\text{in }\Omega,\label{eq:OriginalEquation_psi_PDE}\\
		\frac{r^2(\psi_r^2+\psi_z^2)}{2}+\frac{F(r^2\psi)^2}{2r^2}+\frac{2m\psi_r}{r}+\frac{2m^2}{r^4}-\sigma\kappa&=Q &&\text{on }\partial\Omega_\mathcal{S},\label{eq:OriginalEquation_psi_Bernoulli}\\
		\psi&=\frac{m}{r^2} &&\text{on }\partial\Omega_\mathcal{S},\label{eq:OriginalEquation_psi_top}
	\end{align}
\end{subequations}
where the operator $\partial_r^2+\frac3r\partial_r$ is the radial Laplacian in four dimensions, and due to the constraint $F(0)=0$ no other singular terms appear. Thus, for rigorous investigations, one can think of $\psi$ as a function on $\R^5$ radially symmetric in the first four components. More precisely, for a function $\psi=\psi(r,z)$ on some $\Omega\subset[0,\infty)\times\R$ we denote by $\I\psi$ the function given by
\[\I\psi(x,z)=\psi(|x|,z)\]
and defined on the set $\Omega^\I$, which results from rotating $\Omega$ around the $z$-axis in $\R^5=\{(x,z)\in\R^4\times\R\}$. At the level of $\I\psi$, \eqref{eq:OriginalEquation_psi_PDE} and \eqref{eq:OriginalEquation_psi_top} read
\begin{subequations}\label{eq:PDE_in_5D}
	\begin{align}
		\Delta_5\I\psi&=-\gamma(|x|^2\I\psi)-\frac{1}{|x|^2}(FF')(|x|^2\I\psi)&\text{in }\Omega^\I,\\
		\I\psi&=\frac{m}{|x|^2}&\text{on }\partial\Omega_\mathcal{S}^\I,
	\end{align}	
\end{subequations}
with $\Delta_5$ denoting the Laplacian in five dimensions.

As for \eqref{eq:OriginalEquation_psi_Bernoulli}, we do not have to take a detour and increase the dimension, since in \eqref{eq:OriginalEquation_psi_Bernoulli} no singular term appears, at least whenever the surface does not intersect the symmetry axis.

\subsection{Underlying trivial flow}
For given $\lambda\in\R$, $d>0$, we consider the underlying flow (or trivial solution) $\psi^{\lambda,d}$, defined as
\begin{align}\label{eq:psi^lambda}
	\psi^{\lambda,d}(s)\coloneqq\psi(sd),\quad s\in[0,1],
\end{align}
where $\psi$ is the unique solution of
\begin{subequations}\label{eq:trivial}
	\begin{align}
		\psi_{rr}+\frac3r\psi_r&=-\gamma(r^2\psi)-\frac{1}{r^2}(FF')(r^2\psi)\quad\text{on }(0,d],\\
		\psi(0)&=\lambda,\\
		\psi_r(0)&=0,
	\end{align}
\end{subequations}
existing due to \eqref{ass:regularity}. The corresponding fluid domain is a cylinder with radius $d$, while the equation $\bm{u}=-2\lambda\bm{e}_z$ at $r=0$ for the corresponding fluid velocity serves as the physical interpretation of $\lambda$.

\subsection{Flattening}
In order to do rigorous analysis, we have to transform the domain into a fixed one. For \eqref{eq:PDE_in_5D} we do this by considering the flattening $(x,z)\mapsto(y,z)=(x/\eta(z),z)$ and for \eqref{eq:OriginalEquation_psi_Bernoulli} by $(r,z)\mapsto(s,z)=(r/\eta(z),z)$, assuming -- as we shall always throughout -- that $\eta>0$. Moreover, a more convenient variable to work with is
\[\phi=\bar\psi-\frac{\langle\eta\rangle^2}{\eta^2}\psi^{\lambda,\langle\eta\rangle},\]
where $\bar\psi$ is related to $\psi$ via $\bar\psi(s,z)=\psi(r,z)$, since $\phi$ is subject to Dirichlet boundary conditions on the surface. Indeed, in terms of $\phi$ and with
\[m=m(\lambda,\langle\eta\rangle),\quad m(\lambda,d)\coloneqq d^2\psi^{\lambda,d}(1)\]
the equations read
\begin{subequations}\label{eq:PDE+Dirichlet_flattened}
	\begin{align}
		L^\eta\I\phi&=-\gamma\left(\eta^2|y|^2\left(\I\phi+\frac{\langle\eta\rangle^2}{\eta^2}\I\psi^{\lambda,\langle\eta\rangle}\right)\right)\nonumber\\
		&\phantom{=\;}-\frac{1}{\eta^2|y|^2}(FF')\left(\eta^2|y|^2\left(\I\phi+\frac{\langle\eta\rangle^2}{\eta^2}\I\psi^{\lambda,\langle\eta\rangle}\right)\right)-L^\eta\frac{\langle\eta\rangle^2\I\psi^{\lambda,\langle\eta\rangle}}{\eta^2}&\text{in }\Omega_0^\I,\\
		\I\phi&=0&\text{on }|y|=1,
	\end{align}
\end{subequations}
and (since $\phi=\phi_z=0$ on $s=1$)
\begin{align}\label{eq:Bernoulli_flattened}
	&\frac{\left(\phi_s+\frac{\langle\eta\rangle^2}{\eta^2}\psi^{\lambda,\langle\eta\rangle}_s\right)^2+\eta_z^2\left(\phi_s+\frac{2m(\lambda,\langle\eta\rangle)+\langle\eta\rangle^2\psi^{\lambda,\langle\eta\rangle}_s}{\eta^2}\right)^2}{2}\nonumber\\
	&+\frac{F(m(\lambda,\langle\eta\rangle))^2}{2\eta^2}+\frac{2m(\lambda,\langle\eta\rangle)\left(\phi_s+\frac{\langle\eta\rangle^2}{\eta^2}\psi^{\lambda,\langle\eta\rangle}_s\right)}{\eta^2}+\frac{2m(\lambda,\langle\eta\rangle)^2}{\eta^4}-\sigma\kappa[\eta]=Q\quad\text{ on }s=1.
\end{align}
Above, $\Omega_0\coloneqq [0,1)\times\R$,
\begin{align}\label{eq:kappa}
	\kappa[\eta]\coloneqq\frac{\eta_{zz}}{(1+\eta_z^2)^{3/2}}-\frac{1}{\eta\sqrt{1+\eta_z^2}}
\end{align}
is the mean curvature of the surface of revolution corresponding to $\eta$, and
\[L^\eta\tilde\psi\coloneqq\tilde\psi_{zz}+\frac{1}{\eta^2}\left(\tilde\psi_{y_iy_i}-2\eta\eta_zy_i\tilde\psi_{y_iz}+\eta_z^2y_iy_j\tilde\psi_{y_iy_j}-(\eta\eta_{zz}-2\eta_z^2)y_i\tilde\psi_{y_i}\right);\]
here, repeated indices are summed over. It is straightforward to see that $L^\eta$ is a uniformly elliptic operator, provided $\eta$ is uniformly bounded from below by a positive constant.

\subsection{Fixed solution and functional-analytic setup}\label{sec:FAsetup}
The goal later will be to apply a global implicit function theorem to obtain a continuum of solutions. To this end, let us now state the fixed solution from which the global continuum emanates. We then describe the precise functional-analytic setup and rewrite the equation in the form \enquote{identity plus compact}, which	will be advantageous for the global continuation argument to come later.

First, we already have found in Section \ref{sec:Unduloids} a two-parameter family of solutions: $\psi\equiv0$ and $\eta=\eta^{b,k}$ (and corresponding constants $m$ and $Q$). Let us take any fixed parameter values $b_0>0$, $k_0\in(0,1)$. As a functional-analytic setup we take a fixed $0<\alpha<1$ and introduce the Banach space
\[X\coloneqq\left\{(\eta,\phi)\in C_{\per,\e}^{2,\alpha}(\R)\times C_{\per,\e}^{0,\alpha}(\overline{\Omega_0}):\phi=0\text{ on }s=1\right\},\]
equipped with the canonical norm
\[\|(\eta,\phi)\|_X=\|\eta\|_{C^{2,\alpha}_\per(\R)}+\|\phi\|_{C^{0,\alpha}_\per(\overline{\Omega_0})}.\]
Here, the indices \enquote{$\per$} and \enquote{$\e$} denote $2\pi b_0$-periodicity and evenness (in $z$ with respect to $z=0$). Now, the point $(\lambda,\eta,\phi)=(0,\eta^{b_0,k_0},0)\in\R\times X$ solves \eqref{eq:PDE+Dirichlet_flattened} and \eqref{eq:Bernoulli_flattened}. Indeed, from \eqref{eq:psi^lambda} and \eqref{eq:trivial}, together with \eqref{ass:gamma(0),F(0)} of course, it is evident that $\psi^{0,d}\equiv0$ and thus also $m(0,d)=0$ for any $d>0$. Therefore, $(0,\eta^{b_0,k_0},0)$ solves \eqref{eq:PDE+Dirichlet_flattened}, and also \eqref{eq:Bernoulli_flattened} with $Q=-\sigma\kappa^{b_0,k_0}$. In the following, we shall always consider a \textit{fixed} Bernoulli constant with value
\begin{align}\label{eq:fixed_Q}
	Q=Q^{b_0,k_0}\coloneqq-\sigma\kappa^{b_0,k_0}.
\end{align}
Because of this and \eqref{eq:kappa_monotone}, our main result will be proved by the implicit function theorem and will \textit{not} be a bifurcation result -- other $\eta^{b,k}$ in the \enquote{vicinity} of $\eta^{b_0,k_0}$ have a different period or correspond to a different Bernoulli constant.

Let us now state the reformulation as \enquote{identity plus compact}. First, for
\[(\lambda,\eta,\phi)\in\R\times\U\coloneqq\R\times\{(\eta,\phi)\in X:\eta>0\text{ on }\R\},\]
we let $\A(\lambda,\eta,\phi)=\I^{-1}\varphi$, where $\varphi\in C^{2,\alpha}(\overline{\Omega_0^\I})$ is the unique solution of
\begin{align*}
	L^\eta\varphi&=-L^\eta\frac{\langle\eta\rangle^2\I\psi^{\lambda,\langle\eta\rangle}}{\eta^2}-\gamma\left(\eta^2|y|^2\left(\I\phi+\frac{\langle\eta\rangle^2}{\eta^2}\I\psi^{\lambda,\langle\eta\rangle}\right)\right)\\
	&\phantom{=\;}-\frac{1}{\eta^2|y|^2}(FF')\left(\eta^2|y|^2\left(\I\phi+\frac{\langle\eta\rangle^2}{\eta^2}\I\psi^{\lambda,\langle\eta\rangle}\right)\right)&\text{in }\Omega_0^\I,\\
	\varphi&=0&\text{on }|y|=1.
\end{align*}
Second, we rewrite \eqref{eq:Bernoulli_flattened}, substituting $\phi$ for $\A$:
\begin{align*}
	&\eta-\eta_{zz}\\
	&=\eta-\sigma^{-1}(1+\eta_z^2)^{3/2}\Bigg(\frac{\sigma}{\eta\sqrt{1+\eta_z^2}}+\frac{\left(\A_s+\frac{\langle\eta\rangle^2}{\eta^2}\psi^{\lambda,\langle\eta\rangle}_s\right)^2+\eta_z^2\left(\A_s+\frac{2m(\lambda,\langle\eta\rangle)+\langle\eta\rangle^2\psi^{\lambda,\langle\eta\rangle}_s}{\eta^2}\right)^2}{2}\\
	&\omit\hfill$\displaystyle+\frac{F(m(\lambda,\langle\eta\rangle))^2}{2\eta^2}+\frac{2m(\lambda,\langle\eta\rangle)\left(\A_s+\frac{\langle\eta\rangle^2}{\eta^2}\psi^{\lambda,\langle\eta\rangle}_s\right)}{\eta^2}+\frac{2m(\lambda,\langle\eta\rangle)^2}{\eta^4}-Q\Bigg)$
\end{align*}
on $s=1$. Compared to \cite{ErhardtWahlenWeber22} the difference here is that we add $\eta$ on both sides; this is because we do \textit{not} fix the mean of $\eta$ here and therefore rather invert $1-\partial_z^2$ instead of $\partial_z^2$ in what follows.

Putting everything together, we reformulate \eqref{eq:PDE+Dirichlet_flattened}, \eqref{eq:Bernoulli_flattened} as 
\begin{align}\label{eq:F=0}
	\F(\lambda,\eta,\phi)=0
\end{align}
for $(\lambda,\eta,\phi)\in\R\times\U$, where
\[\F\colon\R\times\U\to X,\;\F(\lambda,\eta,\phi)=(\eta,\phi)-\M(\lambda,\eta,\phi),\]
with $\M=(\M^1,\M^2)$,
\begin{align*}
	\M^1(\lambda,\eta,\phi)&\coloneqq(1-\partial_z^2)^{-1}\Bigg(\eta-\sigma^{-1}(1+\eta_z^2)^{3/2}\Bigg(\frac{\sigma}{\eta\sqrt{1+\eta_z^2}}\\
	&\phantom{\coloneqq\;}+\frac{\left(\SL\A_s+\frac{\langle\eta\rangle^2}{\eta^2}\SL\psi^{\lambda,\langle\eta\rangle}_s\right)^2+\eta_z^2\left(\SL\A_s+\frac{2m(\lambda,\langle\eta\rangle)+\langle\eta\rangle^2\SL\psi^{\lambda,\langle\eta\rangle}_s}{\eta^2}\right)^2}{2}\\
	&\phantom{\coloneqq\;}+\frac{F(m(\lambda,\langle\eta\rangle))^2}{2\eta^2}+\frac{2m(\lambda,\langle\eta\rangle)\left(\SL\A_s+\frac{\langle\eta\rangle^2}{\eta^2}\SL\psi^{\lambda,\langle\eta\rangle}_s\right)}{\eta^2}+\frac{2m(\lambda,\langle\eta\rangle)^2}{\eta^4}-Q\Bigg)\Bigg),\\
	\M^2(\lambda,\eta,\phi)&\coloneqq\A(\lambda,\eta,\phi).
\end{align*}
Here, $\SL$ is just shorthand for the evaluation operator on $s=1$. As in \cite{ErhardtWahlenWeber22}, \eqref{eq:F=0} is equivalent to \eqref{eq:PDE+Dirichlet_flattened}, \eqref{eq:Bernoulli_flattened}, and it is straightforward to see that $\M$ is compact (on subsets of $\R\times X$ on which $\eta$ is uniformly bounded from below by a positive constant) as it \enquote{gains a derivative}.

\section{The global continuum of solutions}\label{sec:MainResult}
The goal of this section is to prove the following main result, consisting of a local and a global statement.
\begin{theorem}\label{thm:MainResult}
	Let $b_0>0$, $k_0\in(0,1)$. Assume \eqref{ass:regularity}, \eqref{ass:gamma(0),F(0)}, and that
	\begin{align}\label{ass:zero_eigenvalue}
		0\text{ is not a Dirichlet eigenvalue of }\Delta^*+r^2\gamma'(0)+F'(0)^2\text{ on }\{(r,z)\in\R^2:0<r<\eta^{b_0,k_0}(z)\}.
	\end{align}
	Then it holds that:
	\begin{enumerate}[label=(\alph*)]
		\item There exists a neighbourhood $(-\delta,\delta)\times V$ of $(0,\eta^{b_0,k_0},0)$ in $\R\times X$ such that for each $\lambda\in(-\delta,\delta)$ there exists a unique point $\nu(\lambda)\in V$ satisfying $\F(\lambda,\nu(\lambda))=0$. Moreover, $\nu\in C^1(-\delta,\delta)$ with $\nu(0)=(\eta^{b_0,k_0},0)$.
		\item Let $\Sigma$ denote the set of all solutions to \eqref{eq:F=0} and $\Cc$ be the (connected) component of $\Sigma$ that contains the local curve of solutions $(\lambda,\nu(\lambda))$, $\lambda\in(-\delta,\delta)$, obtained in (a). Then one of the following three statements is true:
		\begin{enumerate}[label=(\roman*)]
			\item $\Cc$ can be written as $\Cc=\{(0,\eta^{b_0,k_0},0)\}\cup\Cc^+\cup\Cc^-$ with $\Cc^+\cap\Cc^-=\emptyset$ and $\Cc^+$, $\Cc^-$ both unbounded in the sense that for any $\#\in\{+,-\}$ one of the following alternatives occurs:
			\begin{enumerate}[label=(\greek*)]
				\item $\sup_{(\lambda,\eta,\phi)\in\Cc^\#}|\lambda|=\infty$;
				\item $\sup_{(\lambda,\eta,\phi)\in\Cc^\#}\|\eta\|_{C^{2,\alpha}_\per(\R)}=\infty$;
				\item $\sup_{(\lambda,\eta,\phi)\in\Cc^\#}\|r^{-(2\alpha+1)/5}\omega^\vartheta\|_{L^{5/(2-\alpha)}(\tilde\Omega_\eta)}=\infty$, where $\omega^\vartheta=-r\gamma(\Psi)-(FF')(\Psi)/r$ is the angular component of the vorticity and $\tilde\Omega_\eta$ is a $2\pi b_0$-periodic instance of the axially symmetric fluid domain in $\R^3$ corresponding to $\eta$.
			\end{enumerate}
			\item $\Cc\setminus\{(0,\eta^{b_0,k_0},0)\}$ is connected.
			\item $\inf_{(\lambda,\eta,\phi)\in\Cc}\inf_\R\eta=0$, that is, intersection of the surface profile with the cylinder axis occurs.
		\end{enumerate}
	\end{enumerate}
\end{theorem}
While (a) follows from an application of the standard (local) implicit function theorem whose hypotheses we shall prove in the rest of this section, part (b) easily follows once (a) is proved. Indeed, all is needed is that the nonlinear operator $\F$ is amenable to degree methods, which is clearly the case for us since we put $\F$ in the classical form \enquote{identity plus compact} in Subsection \ref{sec:FAsetup}. For details on a general global implicit function theorem we refer to \cite[Chapter II.6]{Kielhoefer}.

The sense in which \enquote{unboundedness} is to be understood in alternative (i) has been sharpened to ($\alpha$)--($\gamma$) by means of a priori estimates. These are exactly the same as in \cite[Proof of Theorem 4]{ErhardtWahlenWeber22}, so we omit the details here.

It would certainly be desirable to rule out alternatives (ii) and (iii) above. Typically, as is well-known in bifurcation theory for water waves, one would need to perform a nodal analysis utilising sharp maximum principles in order to eliminate them. However, in the presence of capillarity such arguments appear to be unavailable. 

So let us now turn to the hypotheses of the standard implicit function theorem. First, it is straightforward to see that $\F$ is of class $C^1$ by virtue of \eqref{ass:regularity}. It therefore remains to prove invertibility of the linearised operator
\[\F_{(\eta,\phi)}(0,\eta^{b_0,k_0},0)\colon X\to X.\]
Let us denote by $\F^1=\eta-\M^1$ and $\F^2=\phi-\M^2$ the two components of $\F$. First we directly observe that $\F^1_\phi(0,\eta^{b_0,k_0},0)=0$, since all terms in $\F^1$ involving $\phi$ are of quadratic order; recall that $\psi^{\cdot,\cdot}=0$, $m=0$, and $\A=0$ at the point $(0,\eta^{b_0,k_0},0)$ at which we differentiate, and that $F(0)=0$. Of course, this is just due to the fact that in the original, physical Bernoulli equation the velocity is \textit{squared}.

Thinking of $\F_{(\eta,\phi)}(0,\eta^{b_0,k_0},0)$ as a $2\times 2$-matrix acting on $(\eta,\phi)$, it is therefore sufficient to show that both diagonal entries $\F^1_\eta(0,\eta^{b_0,k_0},0)$ and $\F^2_\phi(0,\eta^{b_0,k_0},0)$ are invertible on $C_{\per,\e}^{2,\alpha}(\R)$ and $\{\phi\in C_{\per,\e}^{0,\alpha}(\overline{\Omega_0}):\phi=0\text{ on }s=1\}$, respectively. Let us first compute
\[\F^2_\phi(0,\eta^{b_0,k_0},0)\phi=\phi-\A_\phi(0,\eta^{b_0,k_0},0)\phi,\]
with $\varphi=\I[\A_\phi(0,\eta^{b_0,k_0},0)\phi]$ the solution of
\begin{align*}
	L^{\eta^{b_0,k_0}}\varphi&=-\left(\gamma'(0)(\eta^{b_0,k_0})^2|y|^2+F'(0)^2\right)\I\phi&\text{in }\Omega_0^\I,\\
	\varphi&=0&\text{on }|y|=1.
\end{align*}
Clearly, $\F^2_\phi(0,\eta^{b_0,k_0},0)$ also admits the form \enquote{identity plus compact}, whence invertibility reduces to injectivity. But, for general $\gamma'(0),F'(0)\in\R$, this is really an assumption that we have to make. More precisely, we have to assume that
\begin{align}\label{ass:zero_eigenvalue_2}
	0\text{ is not a Dirichlet eigenvalue of }L^{\eta^{b_0,k_0}}+\gamma'(0)(\eta^{b_0,k_0})^2|y|^2+F'(0)^2,
\end{align}
viewed as an operator acting on functions $\varphi=\varphi(y,z)$ defined on $\Omega_0^\I$ and radial in $y$. Translating back to $\Psi$, the assumption \eqref{ass:zero_eigenvalue_2} can also be expressed as \eqref{ass:zero_eigenvalue}.

In general, one has to examine whether this assumption is satisfied, given values of $\gamma'(0),F'(0)\in\R$, and parameter values $b_0>0$, $k_0\in(0,1)$. The following proposition provides some easy observations in this direction.
\begin{proposition}\label{prop:cond_zero_eigenvalue}
	Assumption \eqref{ass:zero_eigenvalue}
	\begin{enumerate}[label=(\alph*)]
		\item always holds in case $\gamma'(0)\le0$, $F'(0)=0$, and therefore, in particular, in the irrotational, swirl-free case $\gamma=F\equiv0$;
		\item is generic in the sense that it holds for almost all values of $(\gamma'(0),F'(0),b_0,k_0)\in\R\times\R\times(0,\infty)\times(0,1)$.
	\end{enumerate}
\end{proposition}
\begin{proof}
	Since the elliptic operator $L^{\eta^{b_0,k_0}}$ contains no zeroth order term, the weak maximum principle ensures (a). Part (b) is then obvious since the elliptic operator in \eqref{ass:zero_eigenvalue_2} depends analytically on $(\gamma'(0),F'(0),b_0,k_0)$.
\end{proof}

Finally, we turn to $\F^1_\eta(0,\eta^{b_0,k_0},0)$. For the sake of clear notation, let us abbreviate $\bar\eta\coloneqq\eta^{b_0,k_0}$. We compute
\[\F^1_\eta(0,\eta^{b_0,k_0},0)\eta=\eta-(1-\partial_z^2)^{-1}\left(\eta+\frac{\bar\eta_z\eta_z}{\bar\eta}+\frac{(1+\bar\eta_z^2)\eta}{\bar\eta^2}-\frac{3\bar\eta_{zz}\bar\eta_z\eta_z}{1+\bar\eta_z^2}\right),\]
recalling \eqref{eq:fixed_Q} and \eqref{eq:kappa}. Again, due to the structure \enquote{identity plus compact} we only have to study the kernel of $\F^1_\eta(0,\eta^{b_0,k_0},0)$ or, equivalently, after applying $1-\partial_z^2$ to the kernel equation, the kernel of the operator
\[T\coloneqq\partial_z^2+\left(\frac{\bar\eta_z}{\bar\eta}-\frac{3\bar\eta_{zz}\bar\eta_z}{1+\bar\eta_z^2}\right)\partial_z+\frac{1+\bar\eta_z^2}{\bar\eta^2}\]
acting on $C_{\per,\e}^{2,\alpha}(\R)$. Of course we can also think of $T$ as simply acting on $C_{\per,\e}^2(\R)$ since $Tv=0$, $v\in C_{\per,\e}^2(\R)$, implies that $v$ is smooth. So we have reduced everything to the study of the kernel of the linear second order ODE operator $T$. This study is the content of the following lemma, from which then Theorem \ref{thm:MainResult} follows. We point out that the proof crucially makes use of the key property \eqref{eq:kappa_monotone}.
\begin{lemma}\label{lma:T_injective}
	$T$ acting on $C_{\per,\e}^2(\R)$ has trivial kernel.
\end{lemma}
\begin{proof}
	When viewed as acting on $C^2(\R)$, the linear second order ODE operator $T$ has two-dimensional kernel. We shall now construct two basis elements of this kernel and prove that one of them is periodic, but odd, and the other is even, but not periodic. From this the statement of the lemma then follows easily.
	
	Now let us recall that by \eqref{eq:kappa} the unduloid profile $\eta^{b,k}$, for any $b>0$, $k\in(0,1)$, satisfies
	\begin{align}\label{eq:unduloid_kappa}
		\frac{\eta^{b,k}_{zz}}{\left(1+(\eta^{b,k}_z)^2\right)^{3/2}}-\frac{1}{\eta^{b,k}\sqrt{1+(\eta^{b,k}_z)^2}}=\kappa^{b,k}=\text{const.}
	\end{align}
	On the one hand, differentiating \eqref{eq:unduloid_kappa} with respect to $z$ at the profile $\bar\eta=\eta^{b_0,k_0}$ gives
	\[\frac{T\bar\eta_z}{\left(1+\bar\eta_z^2\right)^{3/2}}=0.\]
	Thus, the first basis element of the kernel of $T$ is $\bar\eta_z$, which is periodic and odd.
	
	On the other hand, differentiating \eqref{eq:unduloid_kappa} with respect to $b$ and $k$, respectively, at the profile $\bar\eta=\eta^{b_0,k_0}$ gives
	\[\frac{T\partial_b\eta^{b_0,k_0}}{\left(1+\bar\eta_z^2\right)^{3/2}}=\partial_b\kappa^{b_0,k_0}\quad\text{and}\quad\frac{T\partial_k\eta^{b_0,k_0}}{\left(1+\bar\eta_z^2\right)^{3/2}}=\partial_k\kappa^{b_0,k_0},\]
	respectively. From these two relations it easily follows that
	\[Tv=0,\quad v\coloneqq\partial_b\eta^{b_0,k_0}-\frac{\partial_b\kappa^{b_0,k_0}}{\partial_k\kappa^{b_0,k_0}}\partial_k\eta^{b_0,k_0}.\]
	Notice that here we make use of the key property \eqref{eq:kappa_monotone} proved in Lemma \ref{lma:kappa_monotone} in order to be able to divide by $\partial_k\kappa^{b_0,k_0}$. It is clear that $v$ is even since all $\eta^{b,k}$, $b>0$, $k\in(0,1)$, are even. So it remains to show that $v$ is not periodic. First, by differentiating \eqref{eq:eta_bk} with respect to $b$ at $(b,k)=(b_0,k_0)$ we see that
	\[\partial_b\eta^{b_0,k_0}(z)=\frac{\bar\eta(z)-z\bar\eta_z(z)}{b_0}.\]
	Here, $\bar\eta$ is clearly periodic, but $z\bar\eta_z$ is not. Moreover, $\partial_k\eta^{b_0,k_0}$ is certainly periodic since varying $k$ and keeping $b$ fixed does not change the period of the unduloid profile -- this is exactly why we introduced the parameters $b$ and $k$ in Section \ref{sec:Unduloids}. As a consequence, $v$ is not periodic and the proof is complete.
\end{proof}
\begin{remark}\label{rem:k0=0}
	Let us finally revisit the above arguments in case $k_0=0$, that is, when the unduloid is a flat cylinder, while still assuming \eqref{ass:zero_eigenvalue}. Noticing that $\eta^{b_0,0}\equiv b_0$, the ODE operator $T=\partial_z^2+1/b_0^2$ acting on $C_{\per,\e}^2(\R)$ now \emph{does} have a kernel, spanned by $\cos(\cdot/b_0)$. We then see that $\F_{(\eta,\phi)}(0,\eta^{b_0,0},0)$ has one-dimensional kernel spanned by $(\cos(\cdot/b_0),0)$, suggesting a bifurcation from the one-parameter family of trivial solutions. However, a direct computation shows that the transversality condition in the spirit of Crandall--Rabinowitz is violated. Notice anyway that static unduloid configurations nearby, that is, with profile $\eta^{b_0,k}$, $0<k\ll1$, are \emph{not} such bifurcating solutions. This is because, again, $Q=Q^{b_0,0}=\sigma/b_0$ is fixed to its value at the flat cylinder in the functional-analytic setup, while recalling \eqref{eq:kappa_monotone}. Notice also that $\partial_k\kappa^{b_0,0}=0$, and that it is easy to see that both $\cos(\cdot/b_0)$ and $\partial_k\eta^{b_0,0}$ generate the kernel of $T$ and are thus multiples of each other.
\end{remark}

\small{\textbf{Acknowledgements.} This project has received funding from the Swedish Research Council (grant no. 2020-00440). EW and JW were supported by the Swedish Research Council under grant no. 2021-06594 while in residence at Institut Mittag-Leffler in Djursholm, Sweden during the fall semester of 2023.

\bibliographystyle{siam}
\bibliography{ref_unduloid}

\end{document}